\documentclass[12pt,oneside,a4paper]{article}

\usepackage[english]{babel}
\usepackage{amsmath,amssymb,amsfonts}
\usepackage[gen,right]{eurosym}
\usepackage{textcomp,longtable,epsfig,url}
\usepackage{bbm}
\usepackage{graphicx}
\usepackage{subfigure}
\usepackage{hyperref}
\usepackage[intoc]{nomencl}
\usepackage[printonlyused]{acronym}
\usepackage{verbatim}
\usepackage{tabularx}
\usepackage{url}
\usepackage{color}
\usepackage{amssymb,amsmath}
\usepackage{setspace}
\usepackage{scrhack}
\usepackage{floatrow}
\usepackage{listings}
\usepackage{amsthm}
\usepackage{paralist}
\usepackage{xr}
\usepackage{graphicx}
\usepackage{listings}
\usepackage{booktabs}
\usepackage{flafter}
\usepackage{varwidth}
\usepackage[round]{natbib}
\usepackage[a4paper]{geometry}
\usepackage{enumitem}
\usepackage{titling}
\usepackage{algorithm2e}

\makeindex

\newtheorem{definition}{Definition}[section]

\newtheorem{prop}[definition]{Proposition}
\newtheorem{lemma}[definition]{Lemma}
\newtheorem{corollary}[definition]{Corollary}
\newtheorem{theorem}[definition]{Theorem}
\newtheorem{example}[definition]{Example}
\newtheorem{remark}[definition]{Remark}

\newcommand{\E}{{\mathbb E}}

\newcommand{\beqo}{\begin{eqnarray*}}
	\newcommand{\eeqo}{\end{eqnarray*}\noindent}
\newcommand{\beq}{\begin{eqnarray}}
	\newcommand{\eeq}{\end{eqnarray}\noindent}

\renewcommand{\P}{\mathbb{P}}
\newcommand{\R}{\mathbb{R}}
\newcommand{\N}{\mathbb{N}}

\newcommand{\x}{\mathbf{x}}
\newcommand{\y}{\mathbf{y}}

\title{Almost stochastic dominance via optimal transport}

\author{Alfred M\"uller\thanks{Department of Mathematics, University of Siegen, Germany. E-mail: \url{mueller@mathematik.uni-siegen.de}.}
\and Johannes Wiesel\thanks{Department of Mathematics, University of Copenhagen, Denmark. E-mail: \url{wiesel@math.ku.dk}.}}

\date{\today.\\[.3cm]
	}

\begin{document}
	\begin{titlingpage}
		\maketitle
		\begin{abstract}  We study parametric classes of almost stochastic dominance on general Polish
spaces as order relations for probability distributions with a parameter
$\gamma \in [0,1]$. Larger values of $\gamma$ correspond to weaker order
relations: $\gamma=0$ gives classical stochastic dominance $\le_{st}$, whereas
$\gamma=1$ gives a complete preorder based on comparison of expectations of a
fixed increasing function $g$. It is well known that $X \le_{st} Y$ can be
characterized by the existence of a solution to an optimal transport problem
with $\mathrm{OT}_c(X,Y)=0$ for a suitable cost function $c$. We generalize this
idea so that the best possible parameter $\gamma$ for almost stochastic
dominance can be determined from the solution of an optimal transport problem.
Using a generalization of the classical Kantorovich--Rubinstein duality theorem
to quasi-pseudo-metrics, we derive a dual characterization of the order in terms
of expectation comparisons for a parametric class of test functions.
Consequently, our relations are always transitive, in contrast to some other recent
approaches to almost stochastic dominance based on optimal transport. A natural
multivariate approach to almost stochastic dominance, based on classes of test
functions with bounds on partial derivatives, was recently introduced by
M\"uller et al.~(2025). We show that this approach is a special case of our
framework and derive the best possible parameters $\gamma$ for examples
considered there, as well as for other examples from the literature. We also
prove a robustness result showing that, under small perturbations of the
distributions in a Wasserstein-type metric related to the optimal transport
problem, the best possible $\gamma$ increases only slightly.

			{\raggedleft {\bf Keywords:} almost stochastic dominance, optimal transport, stochastic order, quasi-pseudo-metric, robustness}   \\[0.2cm]
					\end{abstract}
	\end{titlingpage}

	\pagestyle{plain}
	
	\onehalfspacing
	\allowdisplaybreaks

\section{Introduction}

The fundamental aim of classical stochastic dominance theory is to provide order relations between probability measures for comparing random quantities; see e.g.~\cite{MueSto:Wiley2002} and \cite{ShaSha:Springe2007} and the references therein for an overview. The most famous example of such an order relation is what is often called \textit{the stochastic order} in probability theory or also \textit{first-order stochastic dominance} in more applied literature. We will write $X \le_{st} Y$ for random elements in an arbitrary partially ordered Polish space, if one of the following two equivalent conditions holds: (i) $\E[u(X)] \le \E[u(Y)]$ for all increasing functions for which the expectations exist, (ii) there exists a coupling $\tilde{X} \overset{d}{=} X$ and $\tilde{Y}\overset{d}{=} Y$ with $\tilde{X} \le \tilde{Y}$ $\P$-a.s. These equivalent formulations explain why stochastic dominance is useful in decision theory, probability, and optimization: if $X \le_{st} Y$, then every rational decision maker who prefers more to less will prefer $Y$ to $X$. This economic interpretation can be found for the univariate case already in \cite{hadar1969}. The general result of the equivalence for arbitrary partially ordered Polish spaces is based on a theorem in \cite{strassen1965} and can be found e.g. in  \cite{kamae1977}. 

A drawback of this formulation is that exact stochastic dominance is often too demanding. Suppose, for instance, that a random payoff is compared with a fixed benchmark. If the random payoff can fall below the benchmark with positive probability, then first-order stochastic dominance of the random payoff over the benchmark is impossible. As another example, consider multivariate normal distributions $\mathbf{X} \sim N(\mu, \Sigma)$, $\mathbf{Y} \sim N(\mu', \Sigma')$ with values in $\mathbb{R}^d$. Then $\mathbf{X} \le_{st} \mathbf{Y}$ only holds if $\mu \le \mu'$ and 
$\Sigma = \Sigma'$, see \cite{Mue:AISM2001}. Clearly the requirement of having the exactly same covariance matrices is very restrictive. 

Parametric families of concepts of almost stochastic dominance with a parameter $\gamma \in [0,1]$ have been introduced to relax these restrictive requirements, where $\gamma = 0$ corresponds to $\le_{st}$ and larger parameters lead to weaker orderings. 
One route is to restrict the class of utility functions, excluding marginal utilities that are too extreme. In the univariate theory of \cite{leshno2002} and \cite{muller2017between}, the relevant classes control the ratio between small and large marginal utilities. 
These approaches have an important structural feature: for each fixed parameter they define genuine transitive order relations, not merely numerical indices of deviation from dominance.

On the real line, the same idea can be described geometrically. For integrable real-valued random variables $X$ and $Y$, let $F$ and $G$ be their distribution functions. The failure of $X\le_{st}Y$ is measured by one of the two one-sided areas between $F$ and $G$, while the compensating part is measured by the other. Using quantiles, these quantities are
\[
D^+(X,Y):=\int_0^1 (F^{-1}(u)-G^{-1}(u))_+\,du,
\qquad
D^-(X,Y):=\int_0^1 (G^{-1}(u)-F^{-1}(u))_+\,du
\]
where we have used the notation $x_+:=\max(x,0)$.
The condition for univariate $\gamma$-almost stochastic dominance can then be written as $D^+(X,Y)\le \gamma D^-(X,Y)$. In the special case where a random payoff is compared with a sure benchmark, this is a gain-loss comparison closely related to the Omega ratio of \cite{ShaKea:JPM2002}.

Optimal transport gives a natural language for these one-sided quantities. For the asymmetric cost $c(x,y)=(x-y)_+$, the violation term is
\[
\mathrm{OT}_c(X,Y):=\inf_{\tilde X\overset{d}{=}X,\ \tilde Y\overset{d}{=}Y}\E[c(\tilde X,\tilde Y)]=D^+(X,Y),
\]
where the infimum is taken over all couplings of $\tilde X$ and $\tilde Y$. The second equality holds as the monotone quantile coupling is optimal for convex costs on the real line; see \cite{cambanis1976}, and \cite{villani2009} for a general introduction to optimal transport. More generally, if $E$ is partially ordered and $c(x,y)=0$ precisely when $x\le y$, then, under the usual closedness and integrability assumptions, the coupling characterization of $X \le_{st} Y$ can be stated in equivalent terms as $\mathrm{OT}_c(X,Y) = 0$. This observation goes back at least to \cite{hansel1978}, who have shown the relation of the theorem in  \cite{strassen1965} to the Ford-Fulkerson theorem. This observation was also the starting point for the definition of the OT-based versions of multivariate almost stochastic dominance recently introduced by \cite{rioux2024}.

The route of defining OT-based versions of almost stochastic dominance has recently been pursued in several forms. The work of \cite{delbarrio2018a} uses Wasserstein-type quantities to assess variants of almost stochastic order, and \cite{rioux2024} formulates a multivariate transport-based criterion for almost stochastic dominance. Closely related statistical and applied formulations appear in \cite{alvarez2017}, and directional decompositions of Wasserstein-type distances are studied in \cite{resin2024}. These contributions show that optimal transport is a flexible way to measure deviations from stochastic dominance and to construct tests or diagnostics for such deviations.

However, measuring a deviation from stochastic dominance is not the same as defining an order relation. If one declares $X$ to be dominated by $Y$ whenever a transport ratio is below a fixed threshold, the resulting threshold relation need not be transitive. This problem already appears in one dimension for ratios based on squared Wasserstein-type costs, as in the transport versions of \cite{delbarrio2018a} and \cite{rioux2024}; see Example \ref{ex:wasserstein-nontransitive} in the Appendix. The same mechanism applies to more general costs generated by convex functions $h$ through $c(x,y)=h(x-y)$. Thus such ratios can be useful measures of distance from stochastic dominance, but they do not by themselves provide partial orders suitable for decision making, see e.g. \cite{alvarez2017} for a detailed discussion of this problem of non-transitivity of decisions based on such tests for stochastic dominance. The exceptional case is the linear absolute-value cost $h(x)=|x|$, where the threshold relation reduces to the classical univariate almost stochastic dominance order of \cite{leshno2002}.

This paper starts from the above distinction. Our goal is not only to quantify how far a pair of laws is from stochastic dominance, but to define an almost dominance relation that remains a partial order, where we can compute the best possible parameter by solving an optimal transportation problem. The modification is simple in the univariate case. Since
\[
\E[Y]-\E[X]=D^-(X,Y)-D^+(X,Y),
\]
the classical ratio $D^+(X,Y)/D^-(X,Y)$ can be rewritten as
\[
\frac{\mathrm{OT}_c(X,Y)}{\E[Y]-\E[X]+\mathrm{OT}_c(X,Y)}
\qquad\text{for } c(x,y)=(x-y)_+.
\]
The denominator is not a symmetric transport distance. It is the violation cost plus the compensating increase in a benchmark functional, here the mean.  This reformulation has the advantage that we only have to solve one optimal transport problem instead of two. This is the feature that survives in the general construction outlined in the main part of the paper and is responsible for transitivity.

We will generalize this idea to the multivariate setting, and even to general partially ordered Polish spaces. The test function characterization of the univariate concept introduced in \cite{leshno2002} can naturally be generalized to the multivariate setting. This has been done in \cite{muller2025multivariate}. They consider as the most natural extension for $\gamma \in [0,1]$ the class 
$\mathcal{U}_\gamma$ of continuously differentiable utilities $u:\mathbb{R}^d\to\mathbb{R}$ satisfying
\[
\gamma\le \frac{\partial u}{\partial x_i}(\x)\le 1
\qquad\text{for all } \x\in\mathbb{R}^d \text{ and } i=1,\dots,d.
\]
Their version of almost stochastic dominance with a parameter $\gamma \in [0,1]$ then holds between random vectors $\mathbf{X}$ and $\mathbf{Y}$, if 
\[
\E[u(\mathbf{X})]\le\E[u(\mathbf{Y})]  \quad\text{for all }u\in\mathcal{U}_\gamma.
\]
\cite{muller2025multivariate} derive sufficient conditions for $\gamma$, but these are far from the best possible $\gamma$ in concrete examples. 

One of the main aims of this paper is to show that the best possible $\gamma$ for this version of almost stochastic dominance can be derived by solving an appropriate optimal transportation problem and defining an equivalent ordering based on the solution of the OT problem. 

\subsection{Structure}

The remainder of the paper is organized as follows. Section 2 recalls quasi-pseudo-metrics, their induced orders, and the coupling characterization of first-order stochastic dominance. Section \ref{sec:main_results} introduces the best possible coefficient $\gamma^\ast$, proves the partial order, duality, and robustness results, and then specializes them to $\mathbb{R}^d$. Section 4 contains two numerical illustrations: day-by-day sunshine measurements for Rome and Siegen as considered in an example by \cite{muller2025multivariate},  and the bivariate Gaussian example of \cite{range2019first}. Section 5 contains the proofs of the main results. The Appendix collects the non-transitivity example for the Wasserstein-ratio criterion and a Banach-lattice example of quasi-pseudo-metrics that explains why we typically choose $p=1$ for the $L^p$-type distances used in our examples.

\subsection{Notation}

Throughout the paper, we fix a Polish space $E$ and denote the set of Borel probability measures on $E$ by $\mathcal{P}(E)$.
All random variables are defined on a joint probability space $(\Omega, \mathcal A, \P)$. For an $E$-valued random variable $X$, we write $\mathrm{Law}(X)$ for its law and set
$$
\mathcal{L}_f(E):=\{X:\E[|f(X)|]<\infty\}.
$$

For a nonnegative Borel measurable cost function $c:E\times E\to [0,\infty)$ and $E$-valued random variables $X,Y$, we write
\[
\mathrm{OT}_c(X,Y):= \inf_{\tilde X\overset{d}{=} X, \tilde Y\overset{d}{=}Y} \E[c(\tilde{X}, \tilde{Y})],
\]
whenever the right-hand side is well-defined. Here we have used $\overset{d}{=}$ to denote equality in law.
For a nonnegative cost function $c:E\times E\to [0,\infty)$ and $E$-valued random variables $X$ we set
\[
\mathcal{L}_c(E):=\{X:\exists x_0\in E \text{ s.t. } \E[c(X,x_0)+c(x_0,X)]<\infty\}.
\]

For $x,y\in \R$ we write $x\vee y:= \max(x,y)$ and $x\wedge y:=\min (x,y)$, as well as $x_+:=x\vee 0$ and $x_-:=(-x)_+$.

\section{Preliminaries}

In this section we collect some basic definitions, that we use throughout the paper. We start with certain generalizations of metrics like quasi-pseudo-metrics and their connection to partial orders as described e.g. in \cite{kelly1963} or \cite{gaba-kuenzi2015,gaba-kuenzi2016}.

\subsection{Metrics and orders}

\begin{definition} \label{def:qpm}
A \emph{quasi-pseudo-metric} is a mapping $d: E \times E \to [0,\infty)$ satisfying the following two properties:
\begin{enumerate}[label=(\roman*)]
\item $d(x,x) = 0$ for all $x \in E$,
\item  $d(x,z) \le d(x,y)  + d(y,z)$ for all  $x,y,z \in E$.
\end{enumerate}
The mapping $d$ is called a \emph{pseudo-metric}, if it satisfies in addition the property
\begin{enumerate}[resume*]
\item $d(x,y) = d(y,x)$ for all  $x,y \in E$,
\end{enumerate}
and the pseudo-metric is a \emph{metric}, if in addition
\begin{enumerate}[resume*]
\item $d(x,y) = 0$ if and only if $x=y$. 
\end{enumerate}
\end{definition}

We next recall the notion of an order on $E$.

\begin{definition} \label{def:order-relation}
For a \emph{relation} $R_\preceq \subseteq E \times E$ we also write $x \preceq y$, if $(x,y) \in R_\preceq$. The relation $\preceq$ is called a \emph{preorder} if it has the following two properties:
\begin{enumerate}[label=(\roman*)]
\item $x \preceq x$ for all $x \in E$,
\item  $x \preceq y$ and $y \preceq z$ implies $x \preceq z$ for all $x,y,z\in E$.
\end{enumerate}
The relation $\preceq$ is called a (partial) \emph{order relation}, if it fulfills in addition the property that
\begin{enumerate}[resume*]
\item $x \preceq y$ and $y \preceq x$ implies $x = y$.
\end{enumerate}
Given a quasi-pseudo-metric $d$ one can define a preorder $\le_d$ via
\begin{align}\label{eq:preorder}
x \le_d y \quad \Longleftrightarrow\quad d(x,y) = 0.
\end{align}
\end{definition}

\begin{remark}
Given a quasi-pseudo-metric $d^+$, define the dual quasi-pseudo-metric $d^-(x,y) := d^+(y,x)$. Next define $\bar{d}(x,y):= d^+(x,y) + d^-(x,y)$, which is a pseudo-metric and indeed a metric, if the corresponding preorder $\le_{d^+}$ is an order relation. The simplest example for this is the case of the classical Euclidean distance of the real numbers 
$$
d(x,y) = |x-y| = d^+(x,y) + d^-(x,y) = \bar{d}(x,y)
$$
with $d^+(x,y) = (x-y)_+$ so that the corresponding partial order is the classical ordering of real numbers: 
$$
x \le_{d^+} y \mbox{ if and only if } x \le y.
$$
A natural extension of this is given by a general Banach lattice as described e.g.~in the book of \cite{schaefer1974}, see Example \ref{ex:schaefer}.
\end{remark}

\subsection{First-order stochastic dominance and Strassen's theorem}

Throughout this section we consider a partially ordered Polish space $(E, \le)$ with closed order relation.
\begin{definition}
For two $E$-valued random variables $X,Y$ we define
\begin{align*}
X \le_{st} Y \quad \Longleftrightarrow \quad \E[f(X)] \le \E[f(Y)] \ \ \text{for all bounded measurable increasing }f:E \to \R.
\end{align*}
\end{definition}

The following result can be found as part of Theorem 1 in \cite{kamae1977} and is an immediate consequence of a famous result of \cite{strassen1965}. 

\begin{prop} \label{prop:kamae}
The following conditions are equivalent for random variables $X,Y$ with values in $(E,\le)$:
\begin{enumerate}[label=(\roman*)]
\item $X \le_{st} Y$,
\item there exists a pair of $E$-valued random variables $\tilde X,\tilde Y$ with $\tilde X \overset{d}{=} X$, $\tilde Y\overset{d}{=} Y$, such that $\tilde X \le \tilde Y$ almost surely.
\end{enumerate}
\end{prop}
This can easily be translated into a statement in terms of optimal transport. 

\begin{prop} \label{prop:kamae-transport}
Let  $c:E \times E \to [0,\infty)$ be a lower semicontinuous function satisfying $c(x,y) = 0 \Leftrightarrow x \le y$. Then the following statements are equivalent for $X,Y\in \mathcal{L}_c(E)$:
\begin{enumerate}[label=(\roman*)]
\item $X \le_{st} Y$,
\item  $\mathrm{OT}_c(X,Y) = 0$.
\end{enumerate}
\end{prop}

\section{A new coefficient for almost stochastic dominance} \label{sec:main_results}

Let us fix a lower semicontinuous quasi-pseudo-metric $c$ on $E$, take a Borel measurable function $g:E\to \R$ and write
\begin{align*}
\mathcal{L}_{c,g}(E)=\{X:\exists x_0\in E \text{ s.t. } \E[c(X,x_0)+c(x_0,X)+|g(X)|]<\infty\}
\end{align*}
for the set of $E$-valued random variables for which $c$ and $g$ are integrable.
We are now in a position to define the central notion of this paper.

\begin{definition}
For $X,Y\in \mathcal{L}_{c,g}(E)$ with $\E[g(X)]\le \E[g(Y)]$ we define 
\begin{align}\label{eq:gamma_star}
\gamma^\ast(X,Y;g)  := \frac{ \mathrm{OT}_c(X,Y)}{\E[g(Y)]-\E[g(X)]+\mathrm{OT}_c(X,Y)},
\end{align}
with the convention $0/0=0$.
\end{definition}

\begin{definition}\label{def:gamma_order}
Let $\gamma\in [0,1]$. For $X,Y\in \mathcal{L}_{c,g}(E)$ we say that $Y$ dominates $X$ in the sense of $(\gamma,g)$-almost stochastic dominance (denoted by $X \le_{\gamma,g} Y$) if $\E[g(X)]\le \E[g(Y)]$ and $\gamma^\ast(X,Y;g)\le \gamma$.
\end{definition}

\subsection{Main results}

We now state our three main results. The first one states that $\le_{\gamma,g}$ defines an order relation for $\gamma<1$ on laws of random variables.

\begin{theorem}\label{thm:main1}
Assume that the function $$ \mathcal{L}_{c,g}(E) \times  \mathcal{L}_{c,g}(E) \ni (X,Y)\mapsto \mathrm{OT}_c(X,Y)+\mathrm{OT}_c(Y,X)$$ is a metric modulo equality in distribution and let $\gamma<1$. Then $\le_{\gamma,g}$ is a partial order on $E$-valued random variables modulo equality in distribution.
\end{theorem}

Obviously the order $\le_{\gamma,g}$ becomes weaker for increasing $\gamma$, i.e.~for $\gamma_1 < \gamma_2$ we obtain 
$$
X \le_{\gamma_1,g} Y \ \Longrightarrow \ X \le_{\gamma_2,g} Y.
$$
If we assume that $g$ is increasing with respect to $\le_c$, then we get a parametric family of orders interpolating between stochastic dominance induced by $\le_c$ for $\gamma=0$ and $\E[g(X)]\le \E[g(Y)]$ for $\gamma=1$. Thus, for $\gamma=1$, we only obtain a preorder in general.

Our second main result states that establishing $X\le_{\gamma,g} Y$ is equivalent to evaluating $X$ and $Y$ on a class of test functions.

\begin{theorem}\label{thm:main2}
For all $\gamma \in [0,1]$ and all $X,Y\in \mathcal{L}_{c,g}(E)$ we have
\begin{align*}
X\le_{\gamma,g} Y \qquad \Longleftrightarrow \qquad \E[f(X)] \le \E[f(Y)] \quad \text{ for all } f \in \mathcal{F}_{\gamma,g}\cup\{g\},
\end{align*}
where $\mathcal{F}_{\gamma,g}:=\{(1-\gamma) f + \gamma g: f(x)-f(y)\le c(x,y) \,\, \text{ for all } x,y\in E\}$.
\end{theorem}

Lastly we have the following robustness result for $\le_{\gamma,g}$. For the special case of univariate almost stochastic dominance as introduced by \cite{leshno2002} a similar robustness result has been shown in \cite{muller2024smps}. 

\begin{theorem}\label{thm:main3}
Assume that $\gamma\in [0,1]$, $X\le_{\gamma,g} Y$ for $X,Y\in \mathcal{L}_{c,g}(E)$ with $\E[g(X)]< \E[g(Y)]$, and that there exists a metric $d$ on $E$ such that $c\le d$ and
\[
|\E[g(\hat X)]-\E[g(\hat Y)]|\le \mathrm{OT}_d(\hat X,\hat Y)\quad \text{ for all } \hat X,\hat Y\in \mathcal{L}_g(E).
\]
If $\mathrm{OT}_d(Y,\tilde Y)\le \epsilon \le \E[g(Y)]-\E[g(X)]$ for some $\tilde{Y}\in \mathcal{L}_{c,g}(E)$, then we have
\begin{align*}
\gamma^\ast(X,\tilde Y;g)\le  \tilde{\gamma} :=  \gamma+\frac{(1-\gamma)\epsilon}{\E[g(Y)]-\E[g(X)]}.
\end{align*}
Analogously, if $\mathrm{OT}_d(X,\tilde X)\le \epsilon \le \E[g(Y)]-\E[g(X)]$ for some $\tilde X\in \mathcal{L}_{c,g}(E)$, then we have $\gamma^\ast(\tilde X, Y;g)\le  \tilde{\gamma}$ for the same $\tilde{\gamma}$.
\end{theorem}

Note that the condition $|\E[g(\hat X)]-\E[g(\hat Y)]|\le \mathrm{OT}_d(\hat X,\hat Y)$ holds for all $g:E\to \R$ with $g(x)-g(y) \le c(x,y)$ $\forall x,y\in E$. The condition $c\le d$ is not a severe restriction. Indeed, for any metric $d'$ we can define
\begin{align*}
d(x,y) := d'(x,y)\vee c(x,y)\vee c(y,x)\qquad x,y\in E
\end{align*}
to obtain a metric satisfying $c\le d$.

\subsection{Almost stochastic dominance on $\R^d$}

We now specialize to $E=\R^d$ and consider the following special case. 
\begin{align}\label{eq:Rd}
c(\x,\y):= \sum_{i=1}^d (x_i-y_i)_+, \quad g(\x):=\sum_{i=1}^d x_i \quad \text{and}\quad d(\x,\y):=  \sum_{i=1}^d |x_i-y_i|,
\end{align}
and recall $\le_{\gamma,g}$ from Definition \ref{def:gamma_order}.
For this metric $d$, write
\[
\mathcal{L}_d(E):=\{\mathbf{X}:\E[d(\mathbf{0},X)]<\infty\}, 
\]
which is simply the set of all random vectors with finite means for all components. We have the following:

\begin{corollary}\label{cor:Rd-specialization}
If $c,g$ and $d$ are chosen according to \eqref{eq:Rd}, then Theorems \ref{thm:main1}-\ref{thm:main3} hold and
\begin{align*}
\mathcal{L}_{c,g}(E) = \mathcal{L}_d(E).
\end{align*}
\end{corollary}

\begin{proof}
Note that $c$ is a quasi-pseudo-metric and $d$ is a metric on $E$.  Furthermore, $c(\x,\mathbf{0})+c(\mathbf{0},\x)=d(\mathbf{0},\x)$ and $|g(\x)|\le d(\mathbf{0},\x)$, so $\mathcal{L}_{c,g}(E)=\mathcal{L}_d(E)$.
Lastly, $(\mathbf{X},\mathbf{Y})\mapsto \mathrm{OT}_c(\mathbf{X},\mathbf{Y})+\mathrm{OT}_c(\mathbf{Y},\mathbf{X})$ is a metric modulo equality in distribution by Lemma \ref{lem:triangle} and Proposition \ref{prop:kamae-transport}.
The claim follows.
\end{proof}

We make the following definition, which is a special case of what is considered in \cite{muller2025multivariate}. Following the same convention, we use the notation $u$ for the test functions, as they often have an interpretation as \textit{utility functions} in applications. 

\begin{definition}\label{def:gamma}
\begin{enumerate}[label=(\roman*)]
\item For $\gamma\in[0,1]$, let $\mathcal{U}_{\gamma}$ be the class of continuously differentiable functions $u \colon \mathbb{R}^{d} \to \mathbb{R}$ such that
\begin{equation}
\label{eq:U-gamma-beta}
\gamma \le \frac{\partial u}{\partial x_i}(\x) \le 1 \quad\text{ for all } \x\in\mathbb{R}^d \text{ and } i\in \{1,\dots, d\}.
\end{equation}
\item We say that $\mathbf{Y}$ dominates $\mathbf{X}$ in $\gamma$-almost stochastic dominance (i.e.~$\mathbf{X}\le_\gamma \mathbf{Y}$), if $\E[u(\mathbf{X})]\le \E[u(\mathbf{Y})]$ for all $u\in \mathcal{U}_\gamma$.
\end{enumerate}
\end{definition}

We can show now the equivalence of this concept of an almost stochastic dominance with the concept based on optimal transport considered here, so that we can determine the best possible parameter $\gamma$ by solving an optimal transport problem. 

\begin{theorem}\label{thm:Rd}
Let $\gamma\in [0,1]$ and assume that $c,g$ and $d$ are given as in \eqref{eq:Rd}.  For $\mathbf{X},\mathbf{Y}\in \mathcal{L}_{d}(\R^d)$ satisfying $\E[g(\mathbf{X})]\le \E[g(\mathbf{Y})]$ we have
\begin{align*}
\gamma^\ast(\mathbf{X},\mathbf{Y};g)\le \gamma \qquad \Longleftrightarrow \qquad \mathbf{X} \le_\gamma \mathbf{Y}.
\end{align*}
\end{theorem}

The robustness result for $\le_\gamma$ obtained from Theorem~\ref{thm:main3} uses, as the distance $OT_d$ between distributions on $\mathbb{R}^{d}$, the Wasserstein metric induced by the distance $d(\x,\y)= \| \x - \y\|_1$ derived from the $\ell_1$-norm. Usually one states robustness result using the classical Wasserstein distance based on Euclidean distance 
\[
W_1(\mathbf{X},\mathbf{Y}) :=\inf_{\tilde{\mathbf{X}} \overset{d}{=} \mathbf{X},\ \tilde{\mathbf{Y}} \overset{d}{=} \mathbf{X}}\E[\|\tilde{\mathbf{X}} - \tilde{\mathbf{Y}}\|_2].
\]
Such a result can still be derived from Theorem \ref{thm:main3}  by using the equivalence of the norms. We obtain the following robustness result of $\le_\gamma$ with respect to the Wasserstein distance $W_1$. 

\begin{corollary}\label{thm:robust-gamma}
Assume that $0 \le \gamma < \gamma' \le 1$, $\mathbf{X} \le_\gamma \mathbf{Y}$ and  $\delta := \E[g(\mathbf{Y})] -  \E[g(\mathbf{X})] > 0$.  If 
\begin{equation} \label{eps-robust}
\varepsilon \le \frac{(\gamma' - \gamma)\delta}{2\sqrt{d}(1-\gamma)},
\end{equation} then it holds for all  $\hat{\mathbf{X}}$ and  $\hat{\mathbf{Y}}$ with $W_1(\mathbf{X},\hat{\mathbf{X}})  \le \varepsilon$ and $W_1(\mathbf{Y},\hat{\mathbf{Y}})  \le \varepsilon$ that
$$
\hat{\mathbf{X}} \le_{\gamma'}  \hat{\mathbf{Y}}.
$$
\end{corollary}

Notice that without incraasing $\gamma$ a bit, we cannot expect any robustness result of such a form, even for very large $\delta$. Indeed, it has been shown by \cite{durante2022}, that for any bounded univariate random variables $X,Y$ with $X \le_{st} Y$ and any $\varepsilon > 0$ one can find a random variable $X'$ with $W_1(X,X') \le \varepsilon$ and $X' \not\le_{st} Y$; see also Theorem 1 in \cite{muller2024smps} for an extension of this non-robustness result to arbitrary unbounded univariate random variables with a finite mean.

\section{Numerical experiments}

In this section we describe some numerical experiments for examples that have already been used in previous papers by \cite{muller2025multivariate}  and \cite{range2019first}.

\subsection{Comparison of sunshine in Rome and Siegen}

For our first example we reconsider the example in section 5 of \cite{muller2025multivariate}, where bounds for $\gamma$ have been derived for the comparison of sunshine distributions of the two cities of Rome and Siegen, where two of the authors of that paper lived. To be more precise, the data used there for the comparison of the intensity of sunshine are data of the so called \textit{global horizontal irradiation} obtained from publicly available satellite data. \cite{muller2025multivariate}  have shown that for this example of 24-dimensional empirical distributions multivariate almost stochastic dominance holds for $\gamma = 0.525$. With the methods of this paper we can now determine the best possible $\gamma$ by solving the appropriate optimal transport problem.

\begin{figure}[htb]
	\centering
	\includegraphics[width=3in]{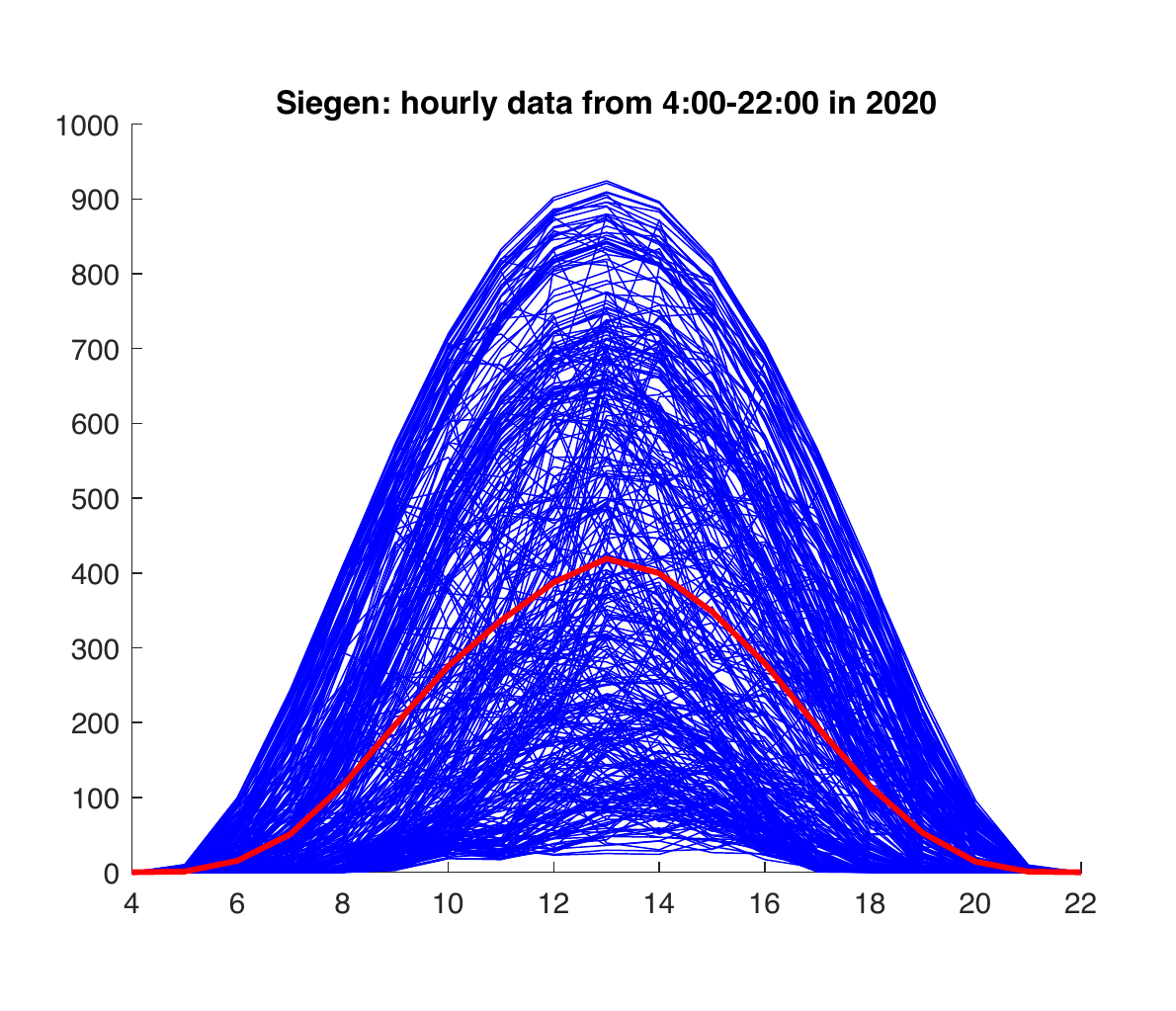}\includegraphics[width=3in]{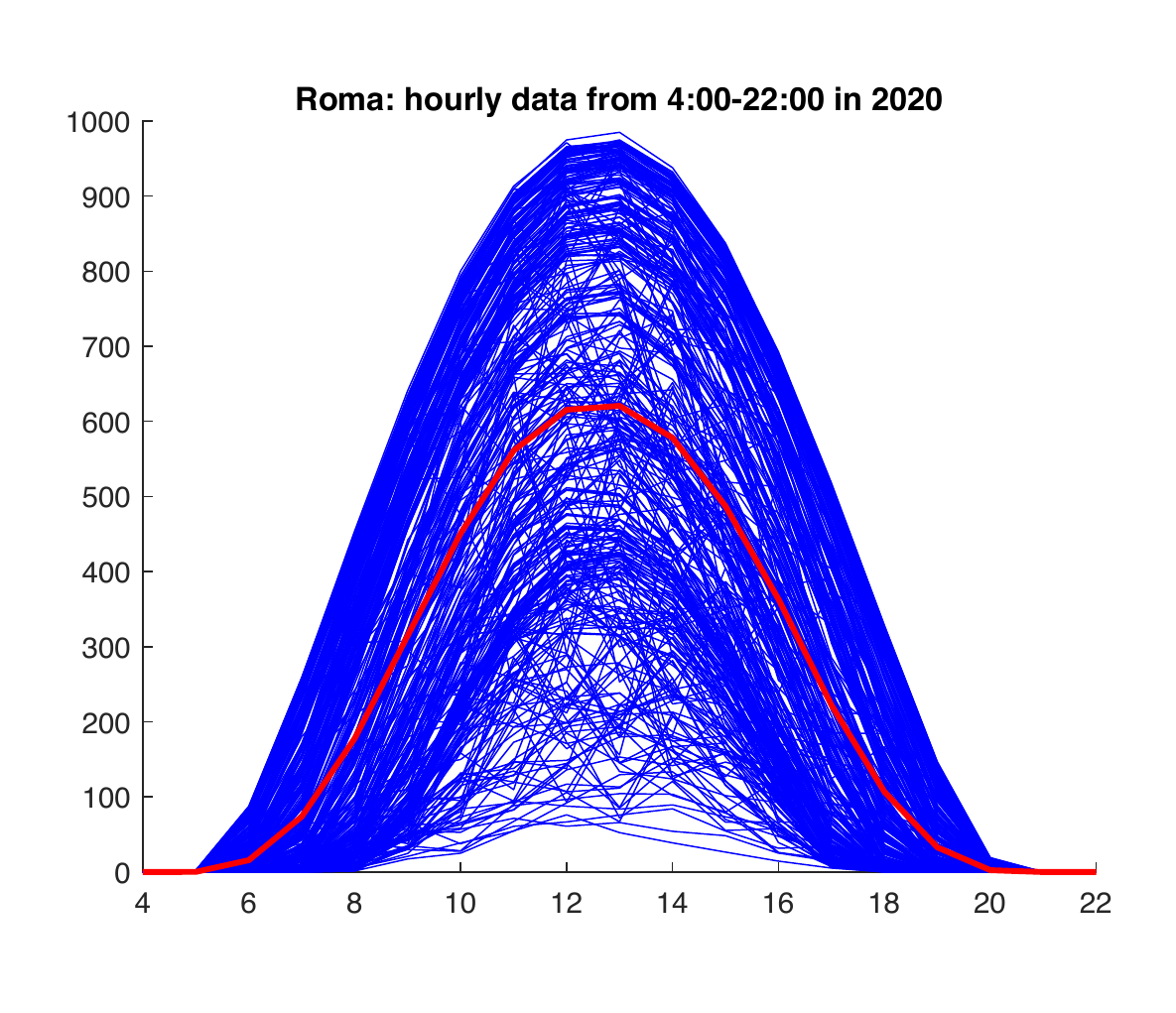}
	\caption{Global horizontal irradiation in Siegen and Rome in 2020.}
	\label{fi:siegen-rome}
\end{figure}

We take hourly measurements of global horiizontal irradiation for Rome and Siegen represented as sets of 24-dimensional points $\{\mathbf{x}_R^i\}_{i = 1}^N$ and $\{\mathbf{x}_S^i\}_{i = 1}^N$ for $N=365$, i.e.~$\mathbf{x}_R,\mathbf{x}_S$ are hourly measurements taken at each day of the year and $d=24$ is the dimension of the state space. The data for the year 2020 are visualized in Figure 1. The black line describes the mean. One can clearly see that on average there is more sunshine in Rome than in Siegen, so that we can find $\gamma < 1$.

We let $\mathbf{X}_R$ and $\mathbf{X}_S$ be random vectors with empirical laws $\mathrm{Law}(\mathbf{X}_R) := \frac{1}{N} \sum_{i = 1}^N \delta_{\mathbf{x}_R^i}$ and $\mathrm{Law}(\mathbf{X}_S) := \frac{1}{N} \sum_{i = 1}^N \delta_{\mathbf{x}_S^i}$ and aim to compute the coefficient
\begin{equation}\label{eqn:asd_coef}
\gamma^\ast(\mathbf{X}_S,\mathbf{X}_R;g) = \frac{\mathrm{OT}_c(\mathbf{X}_S,\mathbf{X}_R)}{\E[g(\mathbf{X}_R)]-\E[g(\mathbf{X}_S)] + \mathrm{OT}_c(\mathbf{X}_S,\mathbf{X}_R)}
\end{equation}
for the cost function $c(\mathbf{x}, \mathbf{y}) = \sum_{i=1}^d (x_i - y_i)_+$ and $g(\x) = \sum_{i = 1}^d x_i$. In this way we determine the threshold $\gamma^\ast$ such that $\mathbf{X}_S \le_{\gamma} \mathbf{X}_R$ holds for any $\gamma \geq \gamma^\ast$ according to Theorem \ref{thm:Rd}. 

\begin{algorithm}[h!]
\caption{Auction algorithm: $\mathcal{X}$ are the ``bidders" and $\mathcal{Y}$ are the ``objects", $b(\mathbf{x}, \mathbf{y})$ represents the benefit of the assignment $\sigma(\mathbf{x})=\mathbf{y}$.}\label{alg:auction}
\KwData{$\mathcal{X} = \{\mathbf{x}^1, \ldots, \mathbf{x}^N\}, \;\; \mathcal{Y} = \{\mathbf{y}^1, \ldots, \mathbf{y}^N\}, \;\; \varepsilon_a > 0$}
\KwResult{$\max_\sigma \sum_{i = 1}^N b(\mathbf{x}^i, \mathbf{y}^{\sigma(i)})$ \text{with absolute error bounded by}  $\varepsilon_a$}
$\varepsilon \gets \max_{i, j} b(\mathbf{x}^i, \mathbf{y}^j), \;\; f: \mathcal{Y} \to \R,\;\; f(\mathbf{y}) \equiv 0$\;
$\sigma: \{1, \ldots, N\} \to \{1, \ldots, N\} \cup \{\emptyset\}, \;\; \sigma(\mathbf{x}) \equiv \emptyset$\;
\While{\text{stopping criterion}($\varepsilon, \varepsilon_a$)}{
\eIf{$\{\mathbf{x}: \sigma(\mathbf{x}) = \emptyset\} = \emptyset$}{$\sigma(\mathbf{x}) \equiv \emptyset$\; $\varepsilon \gets \varepsilon / 5$\;}
{$\mathbf{x} \gets \text{an element of } \{\mathbf{x}:\sigma(\mathbf{x}) = \emptyset\}$\; $\mathbf{y}_{\mathbf{x}} \gets \text{argmax}_{\mathbf{y} \in \mathcal{Y}} [b(\mathbf{x}, \mathbf{y}) - f(\mathbf{y})], \;\; \tilde{\mathbf{y}}_{\mathbf{x}} \gets \text{second best}$\;
$f(\mathbf{y}_{\mathbf{x}}) \gets f(\mathbf{y}_{\mathbf{x}}) + [b(\mathbf{x}, \mathbf{y}_{\mathbf{x}}) - f(\mathbf{y}_{\mathbf{x}}) - b(\mathbf{x}, \tilde{\mathbf{y}}_{\mathbf{x}}) + f(\tilde{\mathbf{y}}_{\mathbf{x}})] + \varepsilon$\;
\If{$\sigma^{-1}(\mathbf{y}_{\mathbf{x}}) \neq \emptyset$}{$\tilde{\mathbf{x}} \gets \sigma^{-1}(\mathbf{y}_{\mathbf{x}}), \;\; \sigma(\tilde{\mathbf{x}}) \gets \emptyset$\;}
$\sigma(\mathbf{x}) \gets \mathbf{y}_{\mathbf{x}}$\;
}
}
\end{algorithm}

To find $\mathrm{OT}_c(\mathbf{X}_S, \mathbf{X}_R)$ we run the Auction Algorithm \ref{alg:auction} \citep{bertsekas1990auction} with $b \equiv -c$ and $\mathcal{X} = \{\mathbf{x}_S^i\}_{i = 1}^N, \mathcal{Y} = \{\mathbf{x}_R^i\}_{i = 1}^N$.  This type of algorithm uses OT duality and is based on the idea of an auction. The aim is to approximate the Kantorovich potential and its $c$-conjugate via a \emph{bidding procedure}: start with some candidate function $f:\mathcal{Y}\to \R$ and think of $f(\mathbf{y})$ as the price of $\mathbf{y}\in \mathcal{Y}$. If $\mathbf{x}$ wants to buy $\mathbf{y}$ (because its price is low), then $\mathbf{y}$ is assigned to $\mathbf{x}$ and the price $f(\mathbf{y})$ of $\mathbf{y}$ is increased just to the point that $\mathbf{y}$ is as desirable for $\mathbf{x}$ as the second best object. One then picks the next unassigned bidder $\mathbf{x}$, assigns another object to $\mathbf{x}$ and raises its price via the same principle. This leads to the following observation: as the prices $f$ can only increase over time, any object $\mathbf{y}$ assigned to $\mathbf{x}$ will stay the best choice for $\mathbf{x}$ until it is assigned to another bidder. The bidding is iterated until all the objects $\mathcal{Y}$ are assigned to the owners $\mathcal{X}$; the optimality of the resulting assignment is then implied by the observation we just made. Importantly for application, the complexity of the Auction algorithm scales \emph{linearly} in $d$. This is the principal reason why we use it here, and why we consider it particularly well suited to high-dimensional applications.

We compute $\gamma^\ast$ for each year separately. The results are shown in Figure \ref{fig:rome-siegen}. According to the data, the coefficient $\gamma^\ast$ is quite stable over the years, with values ranging from 0.11 to 0.16. As expected, this shows that there is no first-order stochastic dominance of $\mathbf{X}_R$ over $\mathbf{X}_S$, and the derived values for $\gamma^\ast$ are much smaller as the value of $\gamma = 0.525$ for the year 2020 that was derived in \cite{muller2025multivariate}. The fact that the best possible values for $\gamma^\ast$ are quite stable over the years is also not very surprising in light of the robustness result Theorem \ref{thm:robust-gamma}.

\begin{figure}[h!]
    \centering
    \includegraphics[scale=0.4]{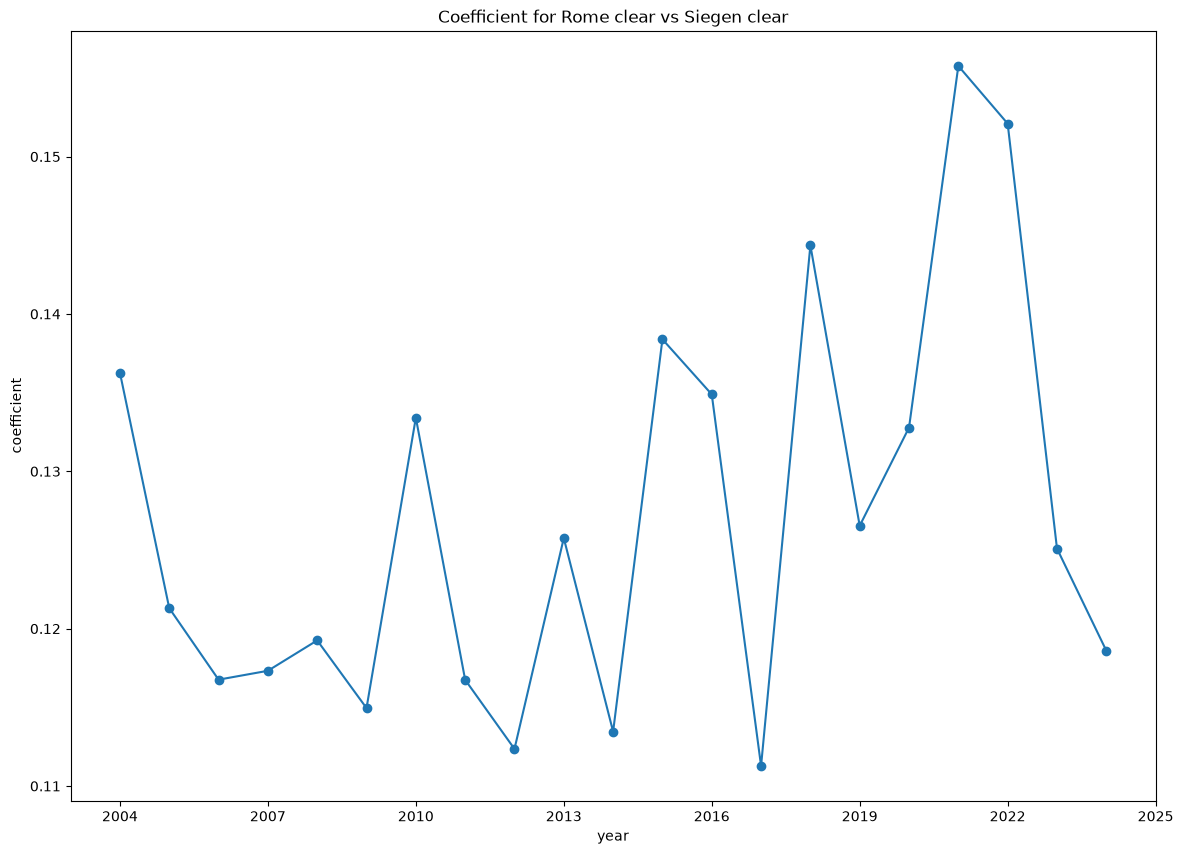}
    \caption{Coefficient $\gamma^\ast$ for yearly sunshine data Rome vs Siegen,  2004-2024}
    \label{fig:rome-siegen}
\end{figure}

\subsection{First-order stochastic dominance for multivariate Gaussian distributions}

We now compute $\gamma^\ast$ for the example in \cite[Section 5]{range2019first}. In this example the authors compare two bivariate normal distributions, which are not in stochastic order. More precisely, let $\mathbf{X}\sim\mathcal{N}(\mu_1, \Sigma_1)$ and $\mathbf{Y}\sim\mathcal{N}(\mu_2, \Sigma_2)$, where 
\begin{align*}
\mu_1= \begin{bmatrix}
    500\\ 500
\end{bmatrix},  \quad
\mu_2= \begin{bmatrix}
    450\\ 450
\end{bmatrix},
\end{align*}
and 
\begin{align*}
\Sigma_1 = \begin{bmatrix}
    15000 & 8000\\ 8000 & 10000
\end{bmatrix} ,  \quad
\Sigma_2= \begin{bmatrix}
    9000 & 5000\\ 5000 & 8000
\end{bmatrix}.
\end{align*}
As $\Sigma_1 \neq \Sigma_2$, the multivariate normal distributions cannot be comparable with respect to $\le_{st}$ as shown in \cite{Mue:AISM2001}. 

As in \cite{range2019first}, we approximate $\mathbf{X},\mathbf{Y}$ by discrete random vectors on a finite lattice and use the cost $c$ and function $g$ from \eqref{eq:Rd}. For this we take $k\in \N$, divide $[1, 1024]\times[1, 1024]$ into $2^{2k}$ boxes of equal length and compute $\gamma^\ast(\mathbf{Y}_k,\mathbf{X}_k;g)$, denoted by $\gamma^\ast(k)$, for the corresponding discrete laws supported on the bottom left corners of each box. Figure \ref{fig:range} shows a plot of $\gamma^\ast(k)$ for $k=1, \dots, 5$. Similarly to \cite{range2019first}, we find that first-order stochastic dominance of $\mathbf{X}_k$ over $\mathbf{Y}_k$ holds (i.e. $\gamma^\ast(k)=0$) for small values of $k\in \N$. However, in our case $\gamma^\ast(k)>0$ for $k\ge 3$, while their test for first-order stochastic dominance only rejects the hypothesis for $k>5$.  The difference comes from the fact, that \cite{range2019first} negect small probabilities of size $< 10^{-6}$ to speed up their algorithm, but our algorithm shows that these small probabilities in the tails matter for this problem. 

\begin{figure}[h!]
    \centering
    \includegraphics[scale=0.3]{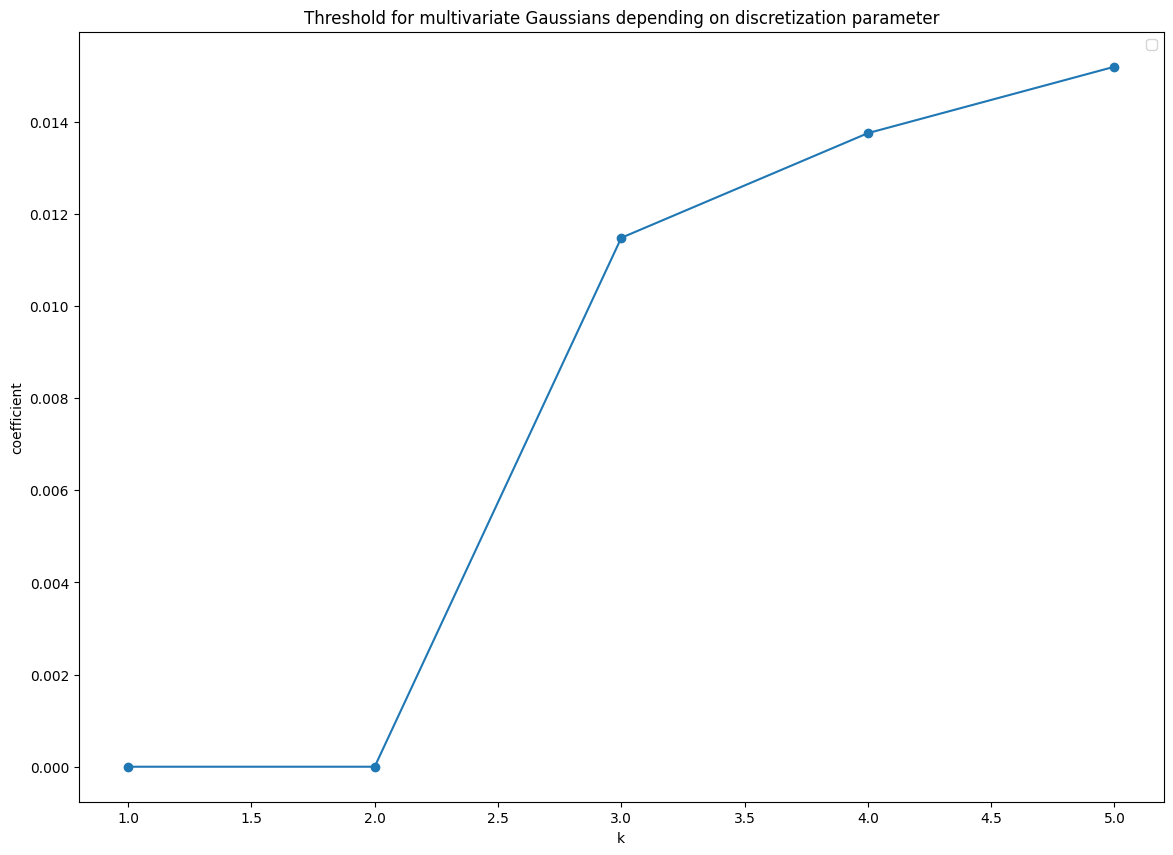}
    \caption{Coefficient $\gamma^\ast$ for discretized bivariate normal distributions}
    \label{fig:range}
\end{figure}

\section{Proofs of main results}

\subsection{Proof of Theorem \ref{thm:main1}}

We need the following lemma.

\begin{lemma} \label{lem:triangle}
If the Borel measurable cost function $c:E\times E \to [0,\infty)$ satisfies the triangle inequality 
\begin{align}\label{eq:triangle}
c(x,z) \le c(x,y) + c(y,z)\quad \mbox{ for all } x,y,z \in E,
\end{align}
then we also have the corresponding triangle inequality for $\mathrm{OT}_c$ for arbitrary $E$-valued random variables $X,Y,Z$:
$$
\mathrm{OT}_c(X,Z) \le \mathrm{OT}_c(X,Y) + \mathrm{OT}_c(Y,Z).
$$
\end{lemma}

\begin{proof}[Proof of Lemma \ref{lem:triangle}]
If $\mathrm{OT}_c(X,Y)+\mathrm{OT}_c(Y,Z)=\infty$, there is nothing to prove. Otherwise, we mimic the proof of \cite[Chapter 6]{villani2009}: we fix $\epsilon>0$ and, by the gluing lemma, take $\tilde X\overset{d}{=} X$, $\tilde Y\overset{d}{=} Y$, and $\tilde Z\overset{d}{=} Z$ such that
$$
|\mathrm{OT}_c(X,Y) -\E[c(\tilde X,\tilde Y)]|\le \epsilon \mbox{ and }  |\mathrm{OT}_c(Y,Z)- \E[c(\tilde Y,\tilde Z)]|\le \epsilon.
$$
Therefore we get
\begin{align*}
\mathrm{OT}_c(X,Z)  &\le  \E[c(\tilde X,\tilde Z)] \stackrel{\eqref{eq:triangle}}{\le} \E[c(\tilde X,\tilde Y)]  + \E[c(\tilde Y,\tilde Z)] \\
&\le \mathrm{OT}_c(X,Y) + \mathrm{OT}_c(Y,Z) +2\epsilon.
\end{align*}
As $\epsilon>0$ was arbitrary, the claim follows.
\end{proof}

\begin{proof}[Proof of Theorem \ref{thm:main1}]
\begin{enumerate}[label=(\roman*)]
\item We obviously have $X \le_{\gamma,g} X$ for every $X\in \mathcal{L}_{c,g}(E)$, as $\mathrm{OT}_c(X,X) = 0$, recalling that $c$ is a quasi-pseudo-metric.

\item  Assume that $X \le_{\gamma,g} Y$ and $Y \le_{\gamma,g} Z$. Then \eqref{eq:gamma_star} can be rewritten as
\begin{align*}
(1-\gamma) \mathrm{OT}_c(X,Y) &\le \gamma[\E[g(Y)] - \E[g(X)]],\\
(1-\gamma) \mathrm{OT}_c(Y,Z) &\le \gamma[\E[g(Z)]-\E[g(Y)]],
\end{align*}
respectively, and hence 
$$
(1-\gamma) \mathrm{OT}_c(X,Z) \stackrel{Lem.\,\ref{lem:triangle}}{\le} (1-\gamma) [\mathrm{OT}_c(X,Y)+\mathrm{OT}_c(Y,Z)] \le \gamma[\E[g(Z)]-\E[g(X)]],
$$
thus $X \le_{\gamma,g} Z$. 

\item If we have $X \le_{\gamma,g} Y$ and $Y \le_{\gamma,g} X$, then necessarily $\E[g(X)] \le \E[g(Y)]$ and $\E[g(Y)] \le \E[g(X)]$, thus $\E[g(X)] = \E[g(Y)]$. Recalling $\gamma<1$, \eqref{eq:gamma_star} can only hold if $\mathrm{OT}_c(X,Y) = 0$ and $\mathrm{OT}_c(Y,X) = 0$, respectively. As $(X,Y)\mapsto \mathrm{OT}_c(X,Y) + \mathrm{OT}_c(Y,X)$ is a metric modulo equality in distribution by assumption, this implies $X\overset{d}{=}Y$.
\end{enumerate}
\end{proof}

\subsection{Proof of Theorem \ref{thm:main2}}

We need a number of auxiliary results.

\begin{lemma} \label{lem:F}
Assume that there is a class $\mathcal{F}$ of functions $f:E\to \R$ such that, for the random variables under consideration,
$$
\mathrm{OT}_c(X,Y) = \sup \{\E[f(X)-f(Y)]: f\in \mathcal{F}\}. 
$$
Then, for $X,Y\in \mathcal{L}_{c,g}(E)$,
\begin{equation} \label{eq:gamma-f-diff}
X \le_{\gamma,g} Y \qquad \Longleftrightarrow \qquad  \E[f(X)] \le \E[f(Y)] \quad \mbox{ for all } f \in \mathcal{F}_{\gamma,g} \cup\{g\},
\end{equation}
where
$$
 \mathcal{F}_{\gamma,g} := \{ (1-\gamma) f + \gamma g: f \in \mathcal{F} \}.
$$
\end{lemma}

\begin{proof}
The test function $g$ gives the condition $\E[g(X)]\le \E[g(Y)]$. Under this condition, apply the assumed representation. Then
\begin{eqnarray*}
 \gamma^\ast(X,Y;g)\le \gamma & \Longleftrightarrow & \mathrm{OT}_c(X,Y) \le \gamma [\E[g(Y)]-\E[g(X)] + \mathrm{OT}_c(X,Y)]\\
& \Longleftrightarrow & (1-\gamma) \mathrm{OT}_c(X,Y) \le \gamma [\E[g(Y)]-\E[g(X)]] \\
& \Longleftrightarrow &   (1-\gamma) \sup_{f \in \mathcal{F}}  \ [\E[f(X)]-\E[f(Y)]] + \gamma [\E[g(X)]-\E[g(Y)]] \le 0 \\
& \Longleftrightarrow  &  (1-\gamma) \E[f(X)] + \gamma \E[g(X)] \le  (1-\gamma) \E[f(Y)] + \gamma \E[g(Y)] \quad \forall f \in \mathcal{F}\\
& \Longleftrightarrow  &   \E[((1-\gamma) f+ \gamma g)(X)]   \le   \E[((1-\gamma) f+ \gamma g)(Y)]   \quad \forall f \in \mathcal{F}.
\end{eqnarray*}
\end{proof}

We now extend the well-known Kantorovich-Rubinstein formula from metrics to quasi-pseudo-metrics, that are not necessarily symmetric. This can be done using the methods described in \cite[Chapter 5]{villani2009}. In the following we will use notations similar to the ones introduced there to describe this extension. 

\begin{definition} \label{c-convex}
a) A function $\psi: E \to \mathbb{R}$ is said to be \emph{$c$-convex}, if there is a function $\zeta: E \to \mathbb{R}$ such that for all $x \in E$ 
$$
\psi(x) = \sup \{ \zeta(y) - c(x,y): y \in E  \}.
$$
b) We define the set of all 1-Lipschitz functions with respect to $c$ as 
$$
\mathcal{L}_1(E,c) := \{ f: E \to \mathbb{R}: f(x) - f(y) \le c(x,y) \mbox{ for all } x,y \in E  \}.
$$
c) For a $c$-convex function we define the \emph{$c$-transform}
$$
\psi^c(x) := \inf \{ \psi(y) + c(y,x): y \in E  \}.
$$
\end{definition}

\begin{lemma} \label{prop-c-convex}
The following hold:
\begin{enumerate}[label=(\roman*)]
\item A function $\psi$ is $c$-convex, if and only if $-\psi \in  \mathcal{L}_1(E,c)$.
\item For a $c$-convex function $\psi$ we have $\psi^c = \psi$.
\item  Any function $\psi \in  \mathcal{L}_1(E,c)$ is $\le_c$-monotone. 
\end{enumerate}
\end{lemma}

\begin{proof}
\begin{enumerate}[label=(\roman*)]
\item As $c$ is a quasi-pseudo-metric, we have $c(z,y) \le c(z,x) + c(x,y)$ for all $x,y,z\in E$, and therefore
$$
c(z,y)-c(x,y) \le c(z,x).
$$
For any $c$-convex function $\psi$ and any $x,z\in E$ we get
\begin{eqnarray*}
\psi(x) - \psi(z) && = \sup \{ \zeta(y)-c(x,y) :y \in E \} - \sup \{ \zeta(y)-c(z,y) :y \in E \}\\
&& \le \sup\{ c(z,y) - c(x,y) :y \in E \} \ \le \ c(z,x)
\end{eqnarray*}
and thus $-\psi \in  \mathcal{L}_1(E,c)$. 

Vice versa, assume that $-\psi \in  \mathcal{L}_1(E,c)$, i.e. $\psi(y)-\psi(x) \le c(x,y)$ for all $x,y\in E$ and hence
$$
\psi(x) \ge \sup\{\psi(y) - c(x,y): y \in E\}.
$$
As $c(x,x) = 0$, we also have 
$
\psi(x) = \psi(x)  - c(x,x) \le \sup\{\psi(y) - c(x,y): y \in E\}
$
and hence $\psi(x) = \sup\{\psi(y) - c(x,y): y \in E\}$, i.e. $\psi$ is $c$-convex. 

\item  If $\psi$ is $c$-convex, then $-\psi \in \mathcal{L}_1$ and hence $\psi(x) \le \psi(y) + c(y,x)$ for all $x,y \in E$. Therefore
$
\psi(x) \le \inf\{\psi(y) + c(y,x): y \in E\} = \psi^c(x).
$
On the other hand
$
\psi^c(x) = \inf\{\psi(y) + c(y,x): y \in E\} \le \psi(x) + c(x,x) = \psi(x),
$
hence $\psi^c = \psi$. 

\item Recall from Definition \ref{def:order-relation} that we have $x \le_c y$, iff $c(x,y) = 0$. Thus, for any $\psi \in \mathcal{L}_1(E,c)$, $x \le_c y$ implies
$
\psi(x) - \psi(y) \le c(x,y) = 0
$
and hence $\psi(x) \le \psi(y)$. 
\end{enumerate}
\end{proof}

We can now state the following generalization of the Kantorovich-Rubinstein formula. 

\begin{lemma} \label{lem:kantorovich-rubinstein-qpm}
For any $E$-valued random variables $X,Y$ such that
$$
\E[(c(X,x_0)+c(x_0,X)+c(Y,x_0)+c(x_0,Y))]<\infty
$$
for some $x_0\in E$, we have
$$
\mathrm{OT}_c(X,Y)  = \sup \{ \E[f(X)-f(Y)]: f \in \mathcal{L}_1(E,c) \}.
$$
\end{lemma}

\begin{proof}
It follows from the well-known Kantorovich duality theorem (see e.g.~\cite[Theorem 5.10]{villani2009}) that 
$$
\mathrm{OT}_c(X,Y)  = \sup \{ \E[\psi^c(Y)- \psi(X)]: \psi \mbox{ is $c$-convex} \}.
$$
As we have shown in Lemma \ref{prop-c-convex}, a function $\psi$ is $c$-convex, if and only if $-\psi \in \mathcal{L}_1(E,c)$, and then $\psi = \psi^c$. This immediately implies the result. 
\end{proof}

\begin{example}
On any partially ordered Polish space $E$ equipped with a closed partial order $\preceq$ we can define the canonical quasi-pseudo-metric $c(x,y) := 1_{[x \not\preceq y]}$ which obviously has the property that $c(x,y) = 0$ if and only if $x \preceq y$. In that case a function $f$ belongs to $\mathcal{L}_1(E,c)$ if and only if $f$ is increasing and $\operatorname{span}(f) := \sup f - \inf f \le 1$. This has already been observed by \cite{KamSta:ORL2020} and they define a \textit{degree of stochastic dominance} $\tau(X,Y) := 1- \mathrm{OT}_c(X,Y)$ with the property that $X \le_{st} Y$ if and only if $\tau(X,Y)=1$ and $\tau(X,Y) < 1$ otherwise. They show that this measure has the following very strong robustness property: if $X \le_{st} Y$ and $X_\gamma,Y_\gamma$ are random variables with
$$
\mathrm{Law}(X_\gamma) = \gamma \mathrm{Law}(X) + (1-\gamma) \mathrm{Law}(X_1), \mbox{ and } \mathrm{Law}(Y_\gamma) = \gamma \mathrm{Law}(Y) + (1-\gamma) \mathrm{Law}(Y_1)
$$
for completely arbitrary $E$-valued random variables $X_1,Y_1$, then $\tau(X_\gamma,Y_\gamma) \ge \gamma$. They call this \textit{partial stochastic dominance}. However, if we define a relation $X \le^\tau_{\gamma,g} Y$ if $\tau(X,Y) \ge \gamma$ then this does not yield an order relation.  Therefore we consider here other variants that lead to order relations.
\end{example}

\begin{proof}[Proof of Theorem \ref{thm:main2}]
By Lemma \ref{lem:kantorovich-rubinstein-qpm} we have
$$
\mathrm{OT}_c(X,Y) = \sup \{\E[f(X)-f(Y)]: f\in \mathcal{L}_1(E,c)\}.
$$
Thus an application of Lemma \ref{lem:F} with $\mathcal{F}=\mathcal{L}_1(E,c)$ yields
\begin{equation*}
X \le_{\gamma,g} Y \qquad \Longleftrightarrow \qquad  \E[f(X)] \le \E[f(Y)] \quad \mbox{ for all } f \in \mathcal{F}_{\gamma,g}\cup\{g\},
\end{equation*}
where
$$
 \mathcal{F}_{\gamma,g} := \{ (1-\gamma) f + \gamma g: f \in \mathcal{L}_1(E,c) \}.
$$
\end{proof}

\subsection{Proof of Theorem \ref{thm:main3}}

\begin{proof}[Proof of Theorem \ref{thm:main3}]
The claim is obvious for $\gamma=1$, so we assume $\gamma<1$ without loss of generality.
Let us define $\delta := \E[g(Y)]-\E[g(X)]>0$. Then 
$$
 \gamma^\ast(X,Y;g)=\frac{\mathrm{OT}_c(X,Y)}{\E[g(Y)]-\E[g(X)] + \mathrm{OT}_c(X,Y)} \le \gamma
$$
if and only if $\mathrm{OT}_c(X,Y) \le \gamma\delta/(1-\gamma)$. Lemma \ref{lem:triangle} implies 
$$
\mathrm{OT}_c(X,\tilde Y) \le \mathrm{OT}_c(X,Y) + \mathrm{OT}_c(Y,\tilde Y)
\le  \frac{\gamma\delta}{1-\gamma} + \mathrm{OT}_d(Y,\tilde Y) \le  \frac{\gamma\delta}{1-\gamma} + \varepsilon.
$$
By assumption we have 
\begin{align*}
\E[g(\tilde Y)]-\E[g(X)] = \E[g(Y)]-\E[g(X)]  + \E[g(\tilde Y)]-\E[g(Y)]  \ge \delta - \varepsilon.
\end{align*}
Therefore
\begin{eqnarray*}
 \gamma^\ast(X, \tilde Y;g) &&=\frac{\mathrm{OT}_c(X,\tilde Y)}{\E[g(\tilde Y)]-\E[g(X)] + \mathrm{OT}_c(X,\tilde Y)} \\
 && \le  \frac{\gamma\delta/(1-\gamma) + \varepsilon}{ \delta - \varepsilon + \gamma\delta/(1-\gamma) + \varepsilon}\\
&& =  \frac{\gamma\delta/(1-\gamma) + \varepsilon}{ \delta + \gamma\delta/(1-\gamma)} =\frac{\gamma+(1-\gamma)\varepsilon/\delta}{(1-\gamma)+\gamma} =: \tilde{\gamma},
\end{eqnarray*}
and note that
$$
\tilde{\gamma} = \gamma + \frac{(1-\gamma)\varepsilon}{\delta} =  \gamma + \frac{(1-\gamma)\varepsilon}{\E[g(Y)]-\E[g(X)]}.
$$
In the case $\mathrm{OT}_d(X,\tilde X) \le \varepsilon$ we get 
\begin{align*}
\mathrm{OT}_c(\tilde X,Y) &\le  \mathrm{OT}_c(X,Y)+ \mathrm{OT}_c(\tilde X,X)\\
&\le    \mathrm{OT}_c(X,Y)+ \mathrm{OT}_d(\tilde X,X)   \le \ \mathrm{OT}_c(X,Y)+ \varepsilon, 
\end{align*}
and the condition on $g$ gives
\[
\E[g(Y)]-\E[g(\tilde X)]
=\delta+\E[g(X)]-\E[g(\tilde X)]\ge \delta-\varepsilon.
\]
Therefore the second result follows from the same arguments. 
\end{proof}

\subsection{Proof of Theorem \ref{thm:Rd}}

\begin{proof}[Proof of Theorem \ref{thm:Rd}]
By Theorem \ref{thm:main2} we have 
\begin{align*}
\gamma^\ast(\mathbf{X}, \mathbf{Y};g) \le \gamma\qquad \Longleftrightarrow \qquad \E[f(\mathbf{X})] \le \E[f(\mathbf{Y})] \quad \text{ for all } f\in \mathcal{F}_{\gamma,g}\cup\{g\},
\end{align*}
where $\mathcal{F}_{\gamma,g}:=\{(1-\gamma) f + \gamma g: f(\x)-f(\y)\le c(\x,\y) \, \forall \x,\y\in \R^d\}$. Thus it is sufficient to show that 
\begin{align*}
\E[f(\mathbf{X})] \le \E[f(\mathbf{Y})] \quad \text{ for all } f\in \mathcal{F}_{\gamma,g} \qquad \Longleftrightarrow\qquad \E[u(\mathbf{X})]\le \E[u(\mathbf{Y})] \quad  \forall u\in \mathcal{U}_\gamma.
\end{align*}
For this we take $u\in \mathcal{U}_\gamma$.
Using the fundamental theorem of calculus along the line segment from $\mathbf{y}$ to $\mathbf{x}$, we have
\begin{eqnarray*}
u(\mathbf{x}) - u(\mathbf{y}) && = \int_0^1 \sum_{i=1}^d (x_i - y_i)\frac{\partial u}{\partial x_i}(\mathbf{y}+t(\mathbf{x}-\mathbf{y}))\,dt\\
&& \le \sum_{i=1}^d [(x_i - y_i)_+ -\gamma (x_i - y_i)_-],  
\end{eqnarray*}
so that 
\begin{align*}
&(u(\x)-\gamma g(\x)) -(u(\y)-\gamma g(\y))\\
&\le \sum_{i=1}^d [(x_i - y_i)_+ -\gamma (x_i - y_i)_-]
-\gamma\sum_{i=1}^d [(x_i - y_i)_+ - (x_i - y_i)_-]\\
&= (1-\gamma) c(\x,\y).
\end{align*}
For $\gamma<1$ this shows that
\[
(u-\gamma g)/(1-\gamma)\in \mathcal{L}_1(E,c),
\]
and hence $u\in \mathcal{F}_{\gamma,g}$.
For $\gamma=1$, the same inequality shows that $u-g$ is constant, so expectation inequalities for $u$ and $g$ are equivalent.
On the other hand, write $f_{\gamma,g}=(1-\gamma)h+\gamma g$ with $h\in \mathcal{L}_1(E,c)$. Then
\[
0\le h(\x+\lambda \mathbf{e}^i)-h(\x)\le \lambda
\quad \forall \lambda>0,\ \x\in\R^d,
\]
where $\mathbf{e}^i=(0,\dots,0,1,0,\dots,0)$ is the $i$th unit vector. Hence $h$ is increasing in each coordinate and 1-Lipschitz with respect to $d$. To show that we can assume without loss of generality $h$ to be differentiable we use the idea of \cite{denuit2002}. Let $h_n$ be standard smooth mollifications of $h$. Then $h_n\to h$ locally uniformly, the functions $h_n$ have a common linear-growth bound, and $0\le \frac{\partial h_n}{\partial x_i}\le 1$ for all $i$. Thus $f_n:=(1-\gamma)h_n+\gamma g\in \mathcal{U}_\gamma$ and, since $\mathbf{X},\mathbf{Y}\in \mathcal{L}_d(\R^d)$,
\[
\E[f_n(\mathbf{X})]\to \E[f_{\gamma,g}(\mathbf{X})]\quad\text{and}\quad \E[f_n(\mathbf{Y})]\to \E[f_{\gamma,g}(\mathbf{Y})].
\]
This proves the reverse implication.
\end{proof}

\subsection{Proof of Corollary \ref{thm:robust-gamma}}

\begin{proof}[Proof of Corollary \ref{thm:robust-gamma}] Assume that $0 \le \gamma < \gamma' \le 1$, $\mathbf{X} \le_\gamma \mathbf{Y}$ and  $\delta := \E[g(\mathbf{Y})] -  \E[g(\mathbf{X})] > 0$. It follows from H\"older's inequality that $\|\mathbf{x}\|_1 \le \sqrt{d}  \|\mathbf{x}\|_2$ for all $\mathbf{x} \in \mathbb{R}^d$ and this inequality carries over to the corresponding Wasserstein distances based on these norms, so that 
$
OT_d(\mathbf{X},\hat{\mathbf{X}})   \le \sqrt{d}  \cdot  W_1(\mathbf{X},\hat{\mathbf{X}})
$.
Therefore we can deduce from Theorem \ref{thm:main3} and Corollary \ref{cor:Rd-specialization} that in case $W_1(\mathbf{X},\hat{\mathbf{X}}) \le \varepsilon$ it holds  $\tilde{\mathbf{X}} \le_{\bar{\gamma}} \mathbf{Y}$ for
$$
\bar{\gamma} = \gamma + (1-\gamma) c , 
$$
where $c = \sqrt{d}\varepsilon/\delta$.  If in addition $W_1(\mathbf{Y},\hat{\mathbf{Y}}) \le \varepsilon$, we get  $\hat{\mathbf{X}} \le_{\hat{\gamma}} \hat{\mathbf{Y}}$ for
$$
\hat{\gamma}= \bar{\gamma} +  (1-\bar{\gamma}) c = \gamma + (1-\gamma) (2c-c^2).
$$
If $\varepsilon$ fulfils \eqref{eps-robust}, we obtain $\hat{\gamma} \le \gamma + 2c (1-\gamma)  \le \gamma'$ and therefore $\hat{\mathbf{X}} \le_{\gamma'} \hat{\mathbf{Y}}$.
\end{proof}

{\large\bf Acknowledgement} 

We kindly thank Ruslan Mirmominov for the implementation of the examples in section 4. Johannes Wiesel would like to thank NNF
and Villum Fonden for their support.

\bibliographystyle{apalike}
\bibliography{asd-transport}
	
\appendix

\section{Appendix}

\begin{example}[A non-transitive Wasserstein-ratio criterion]\label{ex:wasserstein-nontransitive} 
The optimal-transport approach of \cite{delbarrio2018a} measures the deviation from first-order stochastic dominance by the fraction of the squared Wasserstein cost that is spent on order violations. In the univariate notation used here, for random variables $X,Y$ with finite second moments, this coefficient is
\[
\varepsilon_{W_2}(X,Y):=
\frac{\mathrm{OT}_{c_+}(X,Y) }
{\mathrm{OT}_c(X,Y) },
\]
where $c(x,y) = (x-y)^2$ and $c_+(x,y) = (x-y)_+^2$. 
and one declares $X \le^B_\varepsilon Y$ if $\varepsilon_{W_2}(X,Y)\le \varepsilon$. This is the one-dimensional squared-cost instance of the transport-ratio idea also used in \cite{rioux2024} in a multivariate context, where one compares a violation transport cost with a total transport cost. The following example shows that thresholding such a ratio need not define a transitive relation. Consider the following random variables with discrete distributions. 
\begin{align*}
X & \sim  \frac{1}{5}(\delta_{10}+\delta_{20}+\delta_{30}+\delta_{40}+\delta_{50}),\\
Y& \sim \frac{1}{100}\delta_9+\frac{19}{100}\delta_{10}
      +\frac{1}{5}(\delta_{20}+\delta_{30}+\delta_{40})
      +\frac{1}{10}(\delta_{50}+\delta_{51}),\\
Z & \sim \frac{1}{100}\delta_8+\frac{19}{100}\delta_{10}
      +\frac{1}{5}(\delta_{20}+\delta_{30})
      +\frac{1}{10}(\delta_{40}+\delta_{41}+\delta_{50}+\delta_{51}).
\end{align*}
A direct calculation gives
\[
\varepsilon_{W_2}(X,Y)=\varepsilon_{W_2}(Y,Z)=\frac{1}{11},
\qquad
\varepsilon_{W_2}(X,Z)=\frac{1}{6}.
\]
Thus, for $\varepsilon=1/10$, we have $ X\le^B_\varepsilon Y$ and $Y \le^B_\varepsilon Z $, but $X \not\le^B_\varepsilon Z$. Hence this threshold relation is not transitive and therefore does not define an order relation.
\end{example}

\begin{example}\label{ex:schaefer}
Assume that we are given a Banach space $B$ with given norm $\|\cdot\|$ and a partial order $\le$ which is characterized by a closed convex cone $C$ such that $x \le y$ if and only if $y-x \in C$. We assume that the Banach space is a Banach lattice, which means that the ordered space $(B, \le)$ is a lattice with lattice operations $x \vee y = \sup\{x,y\}$ and $x \wedge y = \inf\{x,y\}$ with the additional property that $|x| \le |y|$ implies $\|x\| \le \|y\|$, where $|x| := x \vee (-x)$. In this case $d^+(x,y) := \|(x-y) \vee 0\|$ defines a quasi-pseudo-metric, as $|(x-z) \vee 0| \le |(x-y) \vee 0| + |(y-z) \vee 0|$ and therefore
$$
d^+(x,z) = \|(x-z) \vee 0\| \le \|(x-y) \vee 0\| + \|(y-z) \vee 0\| = d^+(x,y) + d^+(y,z)
$$
for all $x,y,z \in B$. As an $L^p$-space on any measure space $(S, \Sigma, \mu)$ equipped with the usual pointwise ordering is a Banach lattice, we get the result that for any functions $f,g \in L^p(\mu)$ 
$$
d_p^+(f,g) = \left(\int \sup\{f(s)-g(s),0\}^p \mu(ds)\right)^{1/p}
$$
is a quasi-pseudo-metric for any $p \ge 1$. The relation to the classical $L^p$-distance 
$$
d_p(f,g) = \left(\int |f(s)-g(s)|^p \mu(ds)\right)^{1/p}
$$
is given by
$$
d_p(f,g)^p = d_p^+(f,g)^p + d_p^-(f,g)^p,
$$
so that we only get the splitting identity $d_p(x,y) = d_p^+(x,y) + d_p^-(x,y) = \bar{d}(x,y)$ for the special case $p = 1$. For $p > 1$ we only obtain
$$
d_p(f,g) \le d_p^+(f,g) + d_p^-(f,g).
$$
In the special case that $S$ is finite and $\mu$ is the counting measure we obtain the quasi-pseudo-metric
$$
d_p^+(\mathbf{x},\mathbf{y}) = \left(\sum_{i=1}^d (x_i - y_i)_+^p \right)^{1/p}
$$
for Euclidean vectors $\mathbf{x}, \mathbf{y} \in \mathbb{R}^d$. For some results in this paper like Theorem \ref{thm:main3}  it is advantageous to have the splitting identity. Therefore choosing $p = 1$ yields more interesting results in our context compared to e.g. the more often used classical $L^2$- distance.  
\end{example}

\end{document}